\documentclass[12pt]{article}
\usepackage{graphicx}
\usepackage{xcolor}
\usepackage{amssymb,amsfonts,amsthm,amsmath,calligra,tikz}
\usepackage[utf8]{inputenc}
\usepackage{accents}
\usepackage{slashed}
\usepackage{yfonts}
\usepackage{mathrsfs,pifont}
\theoremstyle{definition}
\usepackage[all]{xy}
\usepackage{enumitem}

\usepackage[style=alphabetic]{biblatex}
\def\be{\begin{eqnarray}}
\def\ee{\end{eqnarray}}

\def\matZ{{\mathbb{Z}}}
\def\matR{{\mathbb{R}}}
\def\matQ{{\mathbb{Q}}}
\def\matC{{\mathbb{C}}}
\def\matN{{\mathbb{N}}}
\newcommand{\bA}{\mathsf{A}}

\newcommand{\bT}{\mathsf{T}}
\newcommand{\bK}{\mathsf{K}}

\newcommand{\Lie}{\mathrm{Lie}}

\let\bs\boldsymbol

\def\zz{{\bs z}}

\def\aa{{\bs a}}

\theoremstyle{definition}
\newtheorem{Definition}{Definition}
\newtheorem{Proposition}{Proposition}
\newtheorem{Lemma}{Lemma}
\newtheorem{Corollary}{Corollary}
\newtheorem{Conjecture}{Conjecture}
\newtheorem{Theorem}{Theorem}

\newtheorem{Note}{Note}

\newcommand{\fC}{\mathfrak{C}} 

\newcommand{\dotr}{\mbox{$\boldsymbol{\cdot}$}}

\def\tb {{\cal{V}}}  

\makeatletter
\newcommand{\somespecialrotate}[3][]{%
\begingroup
\sbox\@tempboxa{#3}%
\@tempdima=.5\wd\@tempboxa
\sbox\@tempboxa{\rotatebox[#1]{#2}{\usebox\@tempboxa}}%
\advance\@tempdima by -.5\wd\@tempboxa
\mbox{\hskip\@tempdima\usebox\@tempboxa}%
\endgroup}
\linethickness{1 pt}

\def\qm {{{\textnormal{\textsf{QM}}}}}

\def \vss {\widehat{{\O}}_{{\rm{vir}}}}

\def\zz{{\bs z}}
\def\dd{{\bs d}}

\def\O {{\mathcal{O}}}

\begin{document}
\title{Quasimaps to zero-dimensional $A_{\infty}$-quiver varieties}
\author{H. Dinkins and A. Smirnov}
\date{}
\maketitle
\thispagestyle{empty}
	
\begin{abstract}
We consider the moduli spaces of quasimaps to zero-dimensional	$A_{\infty}$ Nakajima quiver varieties. An explicit combinatorial formula for the equivariant Euler characteristic of these moduli spaces is obtained and applications to symplectic duality are discussed.
\end{abstract}

\maketitle

\section{Introduction}
\subsection{\label{firsec}} 
In this paper we study the equivariant Euler characteristic of certain moduli spaces parameterizing vector bundles over $\mathbb{P}^1$ subject to special stability conditions.
These moduli spaces parametrize {\it quasimaps} $\mathbb{P}^1 \dashrightarrow X_{\lambda}$, where $X_{\lambda}$ is a zero-dimensional Nakajima variety associated with the $A_{\infty}$ quiver.

 These moduli spaces can be defined as follows. Let $\lambda$ be a Young diagram rotated  by $45^{\circ}$ as in Figure \ref{yng1}.  Let $\textsf{v}_i \in  \matN$, $i\in \matZ$ denotes the number of boxes in $i$'s vertical column. We assume that $i=0$ corresponds to the column which contains the corner  box of $\lambda$.

\begin{Definition} \label{maindef}
	{ \it For $\dd=(d_i)$, $i\in \matZ$ the moduli space $\qm^{\dd}_{\lambda}$ is the stack classifying the following data:
		
		$(\star)$ rank $\textsf{v}_i$ vector bundles $\tb_i$ over $\mathbb{P}^{1}$ of degrees $\deg(\tb_i)=d_i$,
		$i \in \matZ$. 
		
		$(\star,\star)$ stable section $s$ of the vector bundle $\mathscr{P} \oplus \mathscr{P}^{*}$ with
		\be  \label{mbn}
		\mathscr{P}= \tb_{0} \oplus \bigoplus\limits_{i\in \matZ} Hom( \tb_{i},  \tb_{i+1})
		\ee  which is non-singular at $\infty \in \mathbb{P}^{1}$ and satisfies the moment map equations.}
\end{Definition} 
\noindent
Let us explain the conditions on the section. The value of a section $s$ at a point $p\in \mathbb{P}^{1}$ provides the following data:
\be \label{pfiber}
s(p)=(I,J,X,Y)
\ee
with vectors $I\in \tb_{0,p}$, $J\in \tb_{0,p}^{*}$ and homomorphisms of vector spaces
$$
X=\bigoplus_{i\in \matZ} X_i, \ \ Y=\bigoplus_{i\in \matZ} Y_i, \ \ \ X_i\in \mathrm{Hom}(\tb_{i,p},\tb_{i+1,p}), \ \ Y_i\in \mathrm{Hom}(\tb_{i+1,p},\tb_{i,p}).
$$
We say that $s(p)$ is {\it stable} if $I$ is a cyclic vector, i.e. 
$$
\bigoplus_{i\in \matZ } \tb_{i,p} = \matC\langle X,Y\rangle I
$$
where $\matC\langle X,Y\rangle$ denotes the ring on (non-commutative) polynomials in $X$ and~$Y$. 

A section  $s$ is called {\it stable} if 
$s(p)$ is stable for all but finitely many points $p$ in $\mathbb{P}^1$.  The exceptional points are called {\it singularities} of $s$. Respectively, $s$ is {\it non-singular} at  $p$ if the point $p$ is not one of the singularities of~$s$.  

Finally, we say that a section $s$ {\it satisfies the moment map equations} if
$$
[X,Y]+I\otimes J=0 \in \bigoplus\limits_{i\in \matZ}\Big( \tb_{i,p} \otimes \tb_{i,p}^{*} \Big), \ \ \ \forall p \in \mathbb{P}^1.
$$

\subsection{} 

The stack $\qm^{\dd}_{\lambda}$ is an example of the moduli space of quasimaps to a GIT quotient constructed and investigated in \cite{qm}. In particular, as shown in \cite{qm} $\qm^{\dd}_{\lambda}$ is a Deligne-Mumford stack of finite type with perfect obstruction theory. We denote by $\vss^{\dd}$
the corresponding symmetrized virtual structure sheaf of~$\qm^{\dd}_{\lambda}$, see Section 2.3 below.

The moduli space $\qm^{\dd}_{\lambda}$ is equipped with a natural action of a two-dimensional torus $\bT:=\matC^{\times}_{q} \times \matC^{\times}_{\hbar}$. The torus $\matC^{\times}_{q}$ acts on $\mathbb{P}^{1}$ by scaling the homogeneous coordinates
$
[x:y]\to [x q: y].
$
The torus $ \matC^{\times}_{\hbar}$ scales the fibers of $\mathscr{P}^{*}$ with character $\hbar^{-1}$, i.e. for every point $p\in \mathbb{P}^1$ the data (\ref{pfiber}) transforms as:
$$
(I,J,X,Y) \to (I ,J\hbar^{-1},X,Y \hbar^{-1}).
$$

It is known that the $\matC_q^{\times}$-fixed locus $(\qm^{\dd}_{\lambda})^{\matC_q^{\times}}$ is proper (see \cite{pcmilect}). In particular, the $\bT$-equivariant Euler characteristic of the structure sheaf 
\be \label{ech}
\chi(\vss^{\dd}) \in K_{\bT}(pt)_{loc} \cong R[\hbar], \ \ R=\matQ(q) 
\ee
is well defined (here $loc$ stands for the $\matC^{\times}_{q}$-localized K-theory).

We introduce the following generating function:
\be \label{partfun}
\mathcal{Z}_{\lambda} = \sum\limits_{\dd}  \, \chi(\vss^{\dd}) \zz^{\dd} \in K_{\bT}(pt)_{loc}[[z]]
\ee
where we denote $\zz^\dd = \prod_i z_i^{d_i}$ for formal parameters $z_i$, usually called {\it K\"ahler parameters}. 

\subsection{} 
Our main result is an explicit combinatorial formula for the generating function $\mathcal{Z}_{\lambda}$. This formula has simplest form in terms of the plethystic exponential.

A virtual $\mathsf{B}$-module is a class $L\in K_{\mathsf{B}}(pt)$, and can be written as $L=L_a-L_b$, where $L_a$ and $L_b$ are $\mathsf{B}$-modules. In this case, if $a_1,\ldots,a_n$ and $b_1,\ldots,b_m$ are the $\mathsf{B}$-weights of $L_a$ and $L_b$, then
\be \label{virp}
L=a_1 + \dots + a_n -b_1 -\dots - b_m \in K_{\mathsf{B}}(pt)
\ee
Assuming $a_i,b_j\neq 1$ for all $i$ and $j$, we let 
$$
S^{\dotr}(L) =\bigoplus\limits_{k=0}^{\infty} S^{k}(L) \in K_{\mathsf{B}}(pt)_{loc}
$$
denote the $\mathsf{B}$-character of the corresponding symmetric algebra of $L$. Explicitly:
\be \label{plet}
S^{\dotr}(L) = \dfrac{(1-b_1)\cdots (1-b_m)}{(1-a_1)\cdots (1-a_n)}.
\ee
The map $S^{\dotr}$ is known as the {\it plethystic exponential}. 
It can be extended to a completion of  $K_{\mathsf{B}}(pt)$ using a suitable norm. This gives a convenient way to write infinite products, for instance:
$$
S^{\dotr}\Big(\dfrac{a}{1-q}\Big)=\prod\limits_{i=0}^{\infty} (1-a q^i)^{-1}
$$
which follows from (\ref{plet}) and the Taylor expansion of $(1-q)^{-1}$. 
See Section 2.1 in \cite{pcmilect} for a detailed discussion of the plethystic exponential.

\subsection{} 
For a box $\square \in \lambda$, let $h_\lambda(\square)$ be the hook in $\lambda$ based at $\square$ and $c(\square)$ be the content of $\square$. Recall, that for a box $\square$ with coordinates $(i,j)$ the content is $c(\square)=i-j$. It can be understood as the horizontal coordinate of $\square$ in Figure \ref{yng1}, normalized so that the corner box has content 0. Define the following element of $\matQ[h^{\pm 1},q^{\pm 1},z_i]$ given by the monomial: 
\begin{equation*}
   z_{\square} := \prod_{\square' \in h_{\lambda}(\square)} \widehat{z}_{c(\square')}
\end{equation*}
where the shifted parameters $\widehat{z}_i$ are
\begin{equation*}
    \widehat{z}_i:=\left(\frac{\hbar}{q}\right)^{\widehat{\sigma}_{\lambda}(i)} z_i \ \ \ \text{where} \ \ \   \widehat{\sigma}_{\lambda}(i):= \begin{cases} 
      \textsf{v}_{i-1}-\textsf{v}_{i} & \text{if} \ \ i \neq 0 \\
 \textsf{v}_{i-1}-\textsf{v}_{i}+1 & \text{if} \ \ i = 0
   \end{cases} 
\end{equation*}
See Figure \ref{yng1}.

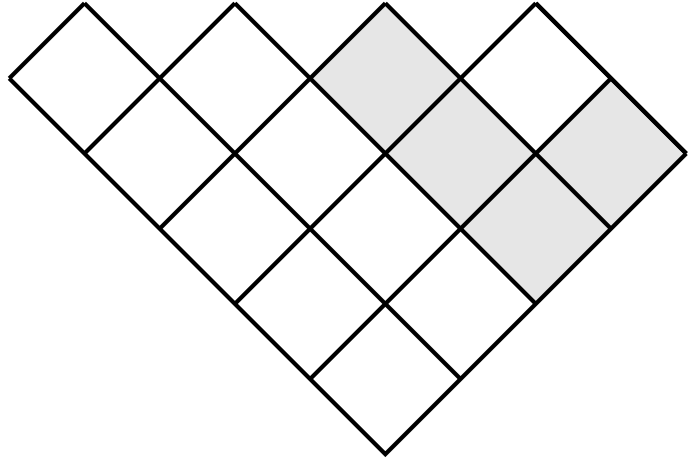
\begin{figure}[ht]
\centering
\begin{tikzpicture}[roundnode/.style={circle, draw=black, very thick, minimum size=5mm},squarednode/.style={rectangle, draw=black, very thick, minimum size=5mm}] 
\draw[ultra thick] (-5,5)--(0,0) -- (4,4);
\draw[ultra thick] (-4,6)--(1,1);
\draw[ultra thick] (-5,5)--(-4,6);
\draw[ultra thick] (-2,6)--(2,2);
\draw[ultra thick] (-4,4)--(-2,6);
\draw[ultra thick]  (0,6)--(3,3);
\draw[ultra thick] (2,6)--(4,4);
\draw[ultra thick] (-1,1)--(3,5);
\draw[ultra thick] (-2,2)--(2,6);
\draw[ultra thick] (-3,3)--(0,6);
\draw[ultra thick] (4,4)--(2,6);
\draw[fill=black,opacity=.1] (2,4)--(0,6)--(-1,5)--(2,2)--(4,4)--(3,5)--cycle;
\end{tikzpicture}
\caption{The partition $\lambda=(5,4,3,2)$ rotated by $45^{\circ}$ and $\textsf{v}=(\textsf{v}_i)=(\ldots,0,0,1,1,2,2,3,2,2,1,0,0,\ldots)$. The shaded boxes illustrate our convention for hooks. If $\square$ is the box at the base of the hook shown, then $z_{\square}= z_0 \left(\frac{\hbar}{q}z_1 \right) z_2 \left(\frac{\hbar}{q}z_3\right)$.} \label{yng1}
\end{figure}
We denote
\be \label{Lchar}
L_{\lambda}:=\sum_{\square \in \lambda} z_{\square} \in \mathbb{Q}[\hbar^{\pm 1},q^{\pm 1}, z_i]
\ee
Our main result is the following formula for $\mathcal{Z}_{\lambda}$. 
\begin{Theorem} \label{mthm}{\it For $|q|<1$ the power series (\ref{partfun}) is the Taylor series expansion of the function
\be\label{mainth}
\mathcal{Z}_{\lambda} = S^{\dotr}\left( \dfrac{1-\hbar}{1-q}  L_{\lambda} \right).
\ee	
holomorphic in the polydisc $|z_{\Box}|<1$. 
}
\end{Theorem}
\noindent
Explicitly, (\ref{mainth}) means that $\mathcal{Z}_{\lambda}$ is an infinite product
\be \label{prform}
\mathcal{Z}_{\lambda} = \prod\limits_{\square \in \lambda} \prod\limits_{i=0}^{\infty} \dfrac{1-\hbar  z_{\square} q^{i}}{1- z_{\square} q^{i}}.
\ee
For $|q|<1$ the product  (\ref{prform}) converges and is clearly holomorphic in variables $z_{\Box}$ in $|z_{\Box}|<1$.  
\begin{Note}
	The formula (\ref{prform}), among other things, implies that $\mathcal{Z}_{\lambda}$ satisfies certain obvious $q$-difference  equations in the K\"ahler parameters. We refer to \cite{OS} for the general theory of these equations. See also the exposition \cite{KorZet}, where these equations are discussed for quiver varieties isomorphic to the cotangent bundles over partial flag varieties of $A$-type.
\end{Note}

\begin{Note}
We note that the formula (\ref{prform}) is similar to (but much simpler than) the formula for the generating function for the equivariant Euler characteristic of the Hilbert scheme of points on $\matC^3$,  conjectured by N.Nekrasov in \cite{NekFor} and then proved in \cite{pcmilect}. The proof of Nekrasov's conjecture can be deduced from factorization of the virtual structure sheaf of the Hilbert scheme of points, see Section 3.5 in \cite{pcmilect}. We believe that (\ref{prform}) can also be proven by a similar geometric argument. 
\end{Note}

\begin{Note}
Note that $X$ defined by (\ref{pfiber}) is a nilpotent endomorphism of $\oplus_i \tb_{i,p}$ corresponding to a partition $\lambda$. The problem of counting vector bundles over Riemann surfaces endowed with a nilpotent endomorphism (over finite fields $\mathbb{F}_q$)  was considered in~\cite{Mellit}.  The partition function in this case can be expressed through special values of Macdonald polynomials, see Theorem~5.5. We expect 
that these formulas should be closely related to ours under a certain specialization of the parameters.
\end{Note}

\section*{Acknowledgments} 

The final version of this paper will appear in \textit{International Mathematics Research Notices}. This work is supported by the Russian Science Foundation under grant 19-11-00062.

\section{Quasimaps to Nakajima Varieties}
\subsection{}
Let $Q$ be a quiver  and $\textsf{v}, \textsf{w} \in \matN^{|Q|}$
be the {\it dimension vectors}. Let $V_{i}, W_i$ be vector spaces with dimensions $\textsf{v}_i,\textsf{w}_i$. The associated representation of the quiver
\be \label{polar}
Rep_{Q}(\textsf{v},\textsf{w})=\bigoplus_{i} Hom(W_i,V_i) \oplus \bigoplus_{i\to j} Hom(V_i,V_j)
\ee
where $i\to j$ denotes the sum over arrows of the quiver, is equipped with a natural action of $G=\prod_i GL(V_i)$. This action extends to a Hamiltonian action of $G$ on $T^{*} Rep_{Q}(\textsf{v},\textsf{w})$. We denote by $\mu: T^{*} Rep_{Q}(\textsf{v},\textsf{w}) \to \mathfrak{g}^{*}$ with $\mathfrak{g}=Lie(G)$ the corresponding moment map. 

The Nakajima quiver variety is defined as the symplectic reduction of this space:
\be \label{qvar}
{\cal{M}}(\textsf{v}, \textsf{w}) = T^{*} Rep_{Q}(\textsf{v},\textsf{w})/\!\!/\!\!/\!\!/_{\!\!\theta} G:=\mu^{-1}(0)^{\theta-ss}/G
\ee
where $\mu^{-1}(0)^{\theta-ss}$ stands for the intersection of the set $\mu^{-1}(0)$ with the locus of semi-stable points
defined by a choice of a character $\theta\in char(G)$, known as the {\it stability parameter}. We refer to \cite{GinzburgLectures} and Chapter 2 in \cite{MO} for details about this construction.

\subsection{} 
For a Young diagram $\lambda$ we consider a Nakajima quiver variety  defined as follows. 
Let $Q$ be the $A_{\infty}$ quiver and let $\textsf{v}=(\textsf{v}_i)$ 
for $i\in \matZ$  be the set of natural numbers defined by $\lambda$ as in Section \ref{firsec}. Let $\textsf{w}_{i}=\delta_{i,0}$. Let  
\be \label{lamqu}
X_{\lambda}:={\cal{M}}(\textsf{v}, \textsf{w})
\ee
denote the quiver variety (\ref{qvar}) defined by these data for the character 
$$
\theta: (g_i) \to \prod_i\det (g_i).
$$
\begin{Proposition}
{\it The quiver variety $X_{\lambda}$ is zero-dimensional (i.e., geometrically $X_{\lambda}$ is a point).}	
\end{Proposition}

\begin{proof}
The dimension  of (\ref{qvar}) equals	
$$
\dim {\cal{M}}(\textsf{v}, \textsf{w}) = 2(\dim Rep_{Q}(\textsf{v},\textsf{w}) - \dim  G)
$$	
For the quiver variety $X_{\lambda}$ this gives
$$
\dim X_{\lambda}=2\left( \textsf{v}_0 + \sum\limits_{i\in \matZ} \textsf{v}_i \textsf{v}_{i+1} - \sum\limits_{i\in \matZ} \textsf{v}_i^2\right).
$$
Elementary combinatorics shows that the last sum is zero for all $\textsf{v}$ defined from a Young diagram. 
\end{proof}

\begin{Note}
	One can show that all zero-dimensional quiver varieties associated with the $A_{\infty}$ quiver (with $A_n$ and $\widehat{A}_{n}$-quivers as special cases) are products of several copies of~(\ref{lamqu}). Informally, one says that zero-dimensional $A_{\infty}$-quiver varieties are labeled by tuples of Young diagrams. 
\end{Note}

\subsection{} 
In Section 4.3  of \cite{pcmilect} A.Okounkov introduces a moduli spaces of quasimaps to a Nakajima quiver variety:
\be \label{qmmod}
\qm^{\dd}_{\textrm{nonsing}\, p}(X)=\{ \textrm{degree $\dd$ quasimaps $\mathbb{P}^1\dashrightarrow X$ nonsingular at } p  \}   
\ee
where $p\in\mathbb{P}^1$. This moduli space is equipped with a natural morphism
\be \label{evmap}
\textrm{ev}_{p}: \qm^{\dd}_{\textrm{nonsing}\, p}(X) \to X
\ee
called {\it the evaluation map}. He then defines the {\it bare vertex function} of $X$ by:
$$
\textbf{V}(\zz)=\sum\limits_{\dd} \textrm{ev}_{p, *}(\vss^{\dd},\qm^{\dd}_{\textrm{nonsing}\, p}(X)) \zz^\dd  \in K_{\bT}(X)_{loc}[[\zz]]
$$
where $\vss^{\dd}$ denotes the symmetrized virtual structure sheaf of (\ref{qmmod}) and $\bT$ denotes a torus acting on $X$. We refer to Section 7.2 in \cite{pcmilect} for details of this construction, definitions, and well-definedness of the maps.

The moduli space from Definition 
\ref{maindef} is an example of a quasimap moduli space.

\begin{Proposition} \label{coinpro}{\it 
The moduli space $\qm^{\dd}_{\lambda}$ is isomorphic to 
$\qm^{\dd}_{\mathrm{nonsing},\infty}(X_{\lambda})$ for $X_{\lambda}$ defined by 	(\ref{lamqu}). }
\end{Proposition} 

\begin{proof}
From the definition of $\qm^{\dd}_{\textrm{nonsing},\infty}$
it is obvious that this stack classifies precisely the data of Definition \ref{maindef}.

	



\end{proof}

\begin{Corollary}
{\it The generating function (\ref{partfun}) is the vertex function of the quiver variety $X_{\lambda}$.}
\end{Corollary}
\begin{proof}
By Proposition (\ref{coinpro}) the evaluation map (\ref{evmap}) is the map to a point, thus $\textrm{ev}_{\infty,*}(\vss^{\dd},\qm^{\dd}_{\textrm{nonsing},\infty}(X_{\lambda}))=\chi(\vss^{\dd})$. \end{proof}

\subsection{}
The vertex functions for the Nakajima quiver varieties are equipped with natural integral representations of Mellin-Barnes type developed in \cite{OkBethe}. We also refer to \cite{Pushk1,Pushk2} where these integral representations were investigated for the quiver varieties isomorphic to cotangent bundles over Grassmannians and partial flag varieties. Here we recall the details of this construction. 

In the case of a type $A_n$ quiver variety $X$ with vertex set given by an interval $I\subset\mathbb{Z}$, dimension vector $\textsf{v}$, and framing dimension vector $\textsf{w}$, we have the following description.

For a character $w_1+ \ldots + w_m \in K_{\bT}(X)$ let us denote
\begin{equation}\label{phidef} \Phi(w_1+\ldots + w_m)=\varphi(w_1)\dots \varphi(w_m), \ \ \ \ \ \varphi(w):=S^{\dotr}\Big(\dfrac{w}{1-q}\Big)=\prod_{n=0}^{\infty}\left( 1-w q^n \right).
\end{equation}
This definition extends by linearity to polynomials with negative coefficients. Let $\mathcal{P}$ be the virtual bundle over a Nakajima variety $X$ associated to the $G$-module
\begin{equation*} 
Rep_{Q}(\textsf{v},\textsf{w})-\mathfrak{g}=\bigoplus_{i \in I} Hom(W_i,V_i) \oplus \bigoplus_{i\to j} Hom(V_i,V_j)- \bigoplus_{i\in I} Hom(V_i,V_i)
\end{equation*}
where the first term is given by (\ref{polar}) and $\mathfrak{g}$ is the adjoint representation of $G$. The subtraction of $\mathfrak{g}$ corresponds to taking the quotient by $G$ in (\ref{qvar}). In the terminology of \cite{MO} (see Section 2.2.7) this class is the canonical polarization of $X$, i.e. the choice of the half of the tangent bundle:
$$
T X =\mathcal{P} + \hbar^{-1} \mathcal{P}^{*} \in K_{\bT}(X).
$$
If $x_{i,1},\dots,x_{i,\mathsf{v}_i}$ denote the Grothendieck roots of $i$-th tautological bundle (i.e. the bundle over $X$ associated to the $G$-module $V_i$) then
\begin{align*}
    \mathcal{P}= \sum_{i \in I} \left( \sum_{j=1}^{\mathsf{v}_i} x_{i,j}^{-1} \right) \left( \sum_{j=1}^{\mathsf{v}_{i+1}} x_{i+1,j} \right) +\sum_{i \in I} \left( \sum_{j=1}^{\textsf{w}_i} a_{i,j}^{-1} \right) \left( \sum_{j=1}^{\mathsf{v}_i} x_{i,j} \right) \\
    - \sum_{i\in I} \left( \sum_{j=1}^{\mathsf{v}_i} x_{i,j}^{-1} \right) \left( \sum_{j=1}^{\mathsf{v}_i} x_{i,j} \right) \in K_\bT(X).
\end{align*}
where $a_{i,j}$ denote the equivariant parameters associated to the framings for $j=1,\ldots, \mathsf{w}_i$.

\noindent
Recall that the vertex functions have natural integral representation  see Section 1.1.5-1.1.6 in \cite{OkBethe}. The integral has the form of a Mellin transform of a certain function  of all Grothendieck roots $\Phi({\bs x},q,\hbar)$ (which can be explicitly expressed in terms of $\mathcal{P}$, see below):
$$
{\bf V}_p(a,z)= \alpha_p \, \int\limits_{C_{p}} \Phi({\bs x},q,\hbar)  {\bf e}({\bs x},\zz) \prod\limits_{i,j} \dfrac{d x_{i,j} }{2 \pi \sqrt{-1} x_{i,j}} 
$$
In this representation, the Grothendieck roots $x_{i,j}$ play a role of coordinates on $\matC^{m}$ where $m=\textrm{rk}(G)$ and the equivariant parameters and K\"ahler parameters are understood as  complex parameters of the integral.  The integral is taken over a cycle $C_p \subset \matC^{m}$ of real dimension $m$ which captures the poles of the function  $\Phi({\bs x},q,\hbar)$
located at the $q$-geometric sequences:
\be \label{poles}
x_{i,j}= x_{i,j}(p) q^{d}, \ \  d=0,1,2 \dots
\ee
where $x_{i,j}(p)$, $j=1,\dots, \mathsf{v}_i$ are the $\bT$-weights of the fiber  $\left.\tb_i\right|_{p}$ of $i$-th tautological bundle at the fixed point $p$. The function ${\bf e}({\bs x},\zz)$ is a $q$-analog of the integral kernel of the Mellin transform:
$$
{\bf e}({\bs x},\zz):=\exp\Big(\dfrac{1}{\ln{q}} \sum\limits_{ i \in I } \ln(z_i)  \ln(\det \mathcal{V}_i) \Big)=\exp\Big(\dfrac{1}{\ln{q}} \sum\limits_{ i \in I }\sum_{j=1}^{\textsf{v}_i} \ln(z_i)  \ln(x_{i,j}) \Big).
$$
For this function we have $\left.{\bf e}({\bs x}  ,\zz)\right|_{x_{i,j}=x_{i,j} q}= z_i {\bf e}( q ,\zz)$ and after picking residues (\ref{poles}) it give rise to a power series in K\"ahler variables $z_{i}$. The prefactor $\alpha_p$ is to normalize the vertex function by
$$
{\bf V}_p(a,z)= 1+ O(z_i)
$$
near $z_i=0$. 

We write $\mathcal{P}_p$ for the character of the virtual bundle $\mathcal{P}$ at a fixed point $p$. In explicit computations we can evaluate the contour integral as a sum over residues (\ref{poles}), which gives the following explicit formula for the vertex function:
\begin{equation} \label{verpow}
{\bf V}_p(\aa,\zz)= \frac{1}{\Phi\Big((q-\hbar) \mathcal{P}_p \Big) {\bf e}({\bs x}(p),\zz)} {\bf \tilde V}_p(\aa,\zz),
\end{equation}
where
\begin{equation}\label{vertild}
\tilde{{\bf V}}_p(\aa,\zz)= \prod\limits_{i,j} \int\limits_{0}^{x_{i,j}(p)}  d_q x_{i,j} \, \Phi\Big((q-\hbar) \mathcal{P} \Big) {\bf e}({\bs x},\zz), 
\end{equation}
the function $\Phi\Big((q-\hbar) \mathcal{P} \Big)$ is defined by (\ref{phidef}) and the $q$-integral stands for the Jackson $q$-integral over all Grothendieck roots, i.e., the infinite sum over $q$-geometric sequences:
$$
\int\limits_{0}^{a}  d_q x f(x) :=\sum\limits_{n=0}^{\infty}\, f(a q^n),
$$
For an indeterminate $x$, we define the $q$-Pochhammer symbol by
\begin{equation*}
(x)_d:=\frac{\varphi(x)}{\varphi(xq^d)} 
\end{equation*}
It is clear that (\ref{verpow}) is a power series in $z_i$ with coefficients which are given by combinations of $q$-Pochhammer symbols, i.e., is a $q$-hypergeometric function.

Let us note that the sum (\ref{verpow}) is exactly the sum over the fixed points $(\qm^{d}_{\lambda})^{\bT}$ which arises when one computes the Euler characteristic (\ref{ech}) by localization in K-theory.  
A fixed point in $(\qm^{d}_{\lambda})^{\bT}$ corresponds bundles in (\ref{mbn}) which decompose into sums of line bundles 
$\mathscr{P}=\bigoplus_k \mathcal{O} (d_k)$.  The $\bT$-character
of the virtual tangent space at this point can be computed by 
Lemma 1 in \cite{Pushk1}. The corresponding contribution to localization formula is given by the $q$-Pochhammer symbols.  

\subsection{}

Let us illustrate formula (\ref{verpow}) in the case of $X_{\lambda}$ for $\lambda=(1)$. As previously remarked, $X_{\lambda}=\{p\}$ is just a point. Here we have only one framing parameter and one Grothendieck root, which we write as $a$ and $x$, respectively. Then
$$
\mathcal{P}= \frac{x}{a} - 1
$$
and :
$$
\Phi\left(\left(q-\hbar \right) \mathcal{P}\right) = \frac{\varphi(q\frac{x}{a})\varphi(\hbar)}{ \varphi(q) \varphi(\hbar \frac{x}{a})} \implies \Phi\left(\left(q-\hbar \right) \mathcal{P}_p\right) = 1 
$$
since the weight of the tautological bundle is $x(p)=a$. So, taking the sum over the $q$-geometric sequence
$$
x=a q^{d}, \ \ d=0,1,2,\dots
$$
we obtain
$$
{\bf V}_p(a,z)= \frac{1}{\Phi\Big((q-\hbar) \mathcal{P}_p \Big) {\bf e}(a,z)} {\bf \tilde V}_p(a,z) = \frac{1}{{\bf e}(a,z)}  \sum_{d=0}^\infty \frac{\varphi(q q^d) \varphi(\hbar)}{\varphi(q) \varphi(\hbar q^d)} {\bf e}(a q^d ,z )
$$
Now, 
$$
{\bf e}(a q^d,z) = \exp\left( \frac{1}{\ln(q)} \ln(z) \ln(a q^d) \right) = z^d {\bf e}(a,z)
$$
and so
$$
{\bf V}_p(a,z)= \sum_{d=0}^\infty \frac{(\hbar)_d}{(q)_d} z^d = 1+ \dfrac{1-\hbar}{1-q} z + \dfrac{(1-\hbar)(1- \hbar q)}{(1-q)(1-q^2) } z^2 +\dots
$$
which is the standard $q$-binomial series
$$
\sum_{d=0}^\infty \frac{(\hbar)_d}{(q)_d} z^d=\prod\limits_{n=1}^{\infty} \dfrac{1-z \hbar q^n}{1-z q^n}
$$
This proves Theorem \ref{mthm} in this simplest example.

\subsection{}
In the case of $X_{\lambda}$ for general $\lambda$, the discussion above provides the following explicit description of the generating function (\ref{partfun}). The weights of the bundle $\mathcal{V}_i$ are 
\begin{equation*}
    \begin{cases} 
    a\hbar^{j-1} & i \geq 0 \\ a\hbar^{j-i-1}& i < 0
    \end{cases}
\end{equation*}
where $a$ is the framing parameter and $j=1,\ldots,\mathsf{v}_i$. Then
\begin{align}\nonumber
  \mathcal{Z}_{\lambda} = \sum_{d_{i,j}} \prod_{j=1}^{\mathsf{v}_0} \frac{(\hbar^j)_{d_{0,j}}}{(q \hbar^{j-1})_{d_{0,j}}} \prod_{i<0} \prod_{j=1}^{\mathsf{v}_i} \prod_{k=1}^{\mathsf{v}_{i+1}} \frac{(\hbar^{k-j})_{d_{i+1,k}-d_{i,j}}}{(q\hbar^{k-j-1})_{d_{i+1,k}-d_{i,j}}}  \\  \label{vf1}
    \prod_{i\geq 0} \prod_{j=1}^{\mathsf{v}_i} \prod_{k=1}^{\mathsf{v}_{i+1}}  \frac{(\hbar^{k-j+1})_{d_{i+1,k}-d_{i,j}}}{(q\hbar^{k-j})_{d_{i+1,k}-d_{i,j}}} \prod_{i\in \mathbb{Z}} \prod_{j,k=1}^{\mathsf{v}_i} \frac{(q\hbar^{k-j})_{d_{i,k}-d_{i,j}}}{(\hbar^{k-j+1})_{d_{i,k}-d_{i,j}}} \, \zz^\dd
\end{align}
where each $d_{i,j}$ is summed from $0$ to $\infty$.

\section{Macdonald Polynomials and Vertex Operators}
\subsection{}
The proof of Theorem \ref{mthm} will use properties of Macdonald polynomials, which we will now introduce. Let $\mathcal{F}= \mathbb{C}[p_1,p_2, \ldots ] \otimes \mathbb{C}(\hbar, q)$ be the space of symmetric polynomials in infinitely many variables, with coefficients in $\mathbb{C}(\hbar,q)$. Following \cite{mac}, we define an inner product on $\mathcal{F}$ by
\begin{equation} \label{inprod}
    \langle p_{\lambda}, p_{\mu} \rangle := \delta_{\lambda,\mu} \prod_{n \geq 1} n^{m_n} m_n! \prod_{i=1}^{l(\lambda)} \frac{1-q^n}{1-\hbar^n} \ \ \  \text{where}  \ \ \  m_n=|\{k \mid \lambda_k=n \}|
\end{equation}
and $l(\lambda)$ is the length of $\lambda$. The standard notation of Macdonald is related to ours by $\hbar=t$.

The Macdonald polynomials $\{M_{\lambda} \}$, indexed by partitions, are the unique basis of $\mathcal{F}$ orthogonal with respect to this inner product satisfying a certain triangularity condition (Chapter 6 in \cite{mac}). In particular, the Macdonald polynomials are uniquely defined by asserting that
$$
\lambda \neq \mu \implies \langle M_{\mu}, M_{\lambda}\rangle = 0
$$
$$
M_{\lambda} = \sum_{\mu \leq \lambda} u_{\lambda\mu} m_{\mu}, \ \ \ u_{\lambda\lambda}=1,  \ \ u_{\lambda,\mu} \in \matQ(q,h)
$$
where $m_\mu$ is the monomial symmetric function corresponding to $\mu$
 and
$$
\mu \leq \lambda \iff \mu_1 + \ldots + \mu_i \leq \lambda_1+\ldots +\lambda_i, \ \ \forall i \geq 0
$$
\subsection{}
We introduce the following vertex operators acting on the completion of $\mathcal{F}$:

\begin{align*}
    \Gamma_{+}(z) &:=\exp \left( \sum_{n=1}^{\infty} \frac{1-\hbar^n}{1-q^n}\frac{z^n}{n}p_n \right) \\
    \Gamma_{-}(z) &:=\exp \left( \sum_{n=1}^{\infty} \frac{1}{z^n}\frac{\partial}{\partial p_n} \right) \\
    z^L \cdot M_\mu &:= z^{|\mu|} M_\mu
\end{align*}
where $z$ is any Laurent monomial in the variables $z_i,\hbar,q$. The inner product (\ref{inprod}) can be interpreted as
\begin{equation}
    \langle p_{\lambda}, p_{\mu} \rangle = \delta_{\lambda,\mu} \left( \prod_{i=1}^{l(\lambda)} \lambda_i \frac{1-q^{\lambda_i}}{1-\hbar^{\lambda_i}} \frac{\partial}{\partial p_{\lambda_i}} \right) p_{\mu}
\end{equation}
which shows that $\Gamma_{-}(\frac{1}{z})$ is the adjoint of $\Gamma_{+}(z)$. 
We have the following:
\begin{Proposition}[\cite{mac} Section 6.6, \cite{glo} Theorem 1.2] Let $(n)$ denote the partition $(n,0,\ldots,0)$. Then
\begin{equation}\label{pieri}
    M_{(n)} M_{\lambda} = \frac{(q)_n}{(\hbar)_n} \sum_{\substack{\mu \succ \lambda \\ |\mu|-|\lambda|=n}} c_{\mu/\lambda} M_{\mu}
\end{equation}
where $\mu \succ \lambda$ means that $\mu_1 \geq \lambda_1 \geq \mu_2 \geq \lambda_2 \geq \ldots$, which we refer to by saying that $\mu$ \textit{interlaces} $\lambda$ \textit{from above}. Here, 
\begin{equation*}
      c_{\mu / \lambda} = \prod_{1\leq i \leq j \leq l(\mu)} \frac{(\hbar^{j-i+1})_{\mu_i-\lambda_j}}{(q\hbar^{j-i})_{\mu_i-\lambda_j}} \frac{(q\hbar^{j-i})_{\mu_i-\mu_j}}{(\hbar^{j-i+1})_{\mu_i-\mu_j}} \prod_{1\leq i<j \leq l(\mu)} \frac{(\hbar^{j-i})_{\lambda_i-\mu_j}}{(q\hbar^{j-i-1})_{\lambda_i-\mu_j}} \frac{(q \hbar^{j-i-1})_{\lambda_i-\lambda_j}}{(\hbar^{j-i})_{\lambda_i-\lambda_j}}
\end{equation*}
where we adopt the convention that $\lambda_k=0$ for $k> l(\lambda)$.
\end{Proposition}

It is known that $\Gamma_{+}(z)$ is equivalent to multiplication by the following infinite sum (see \cite{mac} Section 6.2):
\begin{align}\label{gammasum}
        \Gamma_{+}(z) = \sum_{n\geq 0} \frac{(\hbar)_n}{(q)_n} M_{(n)} z^n 
\end{align}
So we can write the action of $\Gamma_+(z)$ in (\ref{gammasum}) as
\begin{equation}\label{pieriplus}
    \Gamma_{+}(z) M_{\lambda} = \sum_{\mu \succ \lambda} c_{\mu / \lambda} z^{|\mu|-|\lambda|} M_{\mu}  
\end{equation}
Since $\Gamma_{-}(\frac{1}{z})$ is the adjoint of $\Gamma_{+}(z)$, we can use this to obtain a formula for the action of $\Gamma_{-}(z)$ on Macdonald polynomials:
\begin{Proposition} Let $||M_\lambda||^2=\langle M_\lambda, M_\lambda \rangle$. Then
\begin{align} \nonumber
 \Gamma_{-}(z) M_{\lambda} &= \sum_{\mu \prec \lambda} \frac{||M_{\lambda}||^2}{||M_{\mu}||^2} c_{\lambda/\mu} M_{\mu} z^{|\mu|-|\lambda|} \\ \label{pieriminus}
    &= \sum_{\mu \prec \lambda} d_{\lambda/\mu} M_{\mu} z^{|\mu|-|\lambda|}
\end{align}
where
\begin{align*}
    d_{\lambda/ \mu}= \prod_{i=1}^{l(\mu)} \frac{(\hbar^{l(\lambda)-i+1})_{\mu_i}}{(q \hbar^{l(\lambda)-i})_{\mu_i}} \prod_{i=1}^{l(\lambda)} \frac{(q \hbar^{l(\lambda)-i})_{\lambda_i}}{(\hbar^{l(\lambda)-i+1})_{\lambda_i}}  \prod_{1 \leq i \leq j \leq l(\lambda)} \frac{(\hbar^{j-i+1})_{\lambda_i-\mu_j}}{(q\hbar^{j-i})_{\lambda_i-\mu_j}} \frac{(q\hbar^{j-i})_{\mu_i-\mu_j}}{(\hbar^{j-i+1})_{\mu_i-\mu_j}} \\
    \prod_{1\leq i < j \leq l(\lambda)} \frac{(q\hbar^{j-i-1})_{\lambda_i-\lambda_j}}{(\hbar^{j-i})_{\lambda_i-\lambda_j}} 
 \frac{(\hbar^{j-i})_{\mu_i-\lambda_j}}{(q\hbar^{j-i-1})_{\mu_i-\lambda_j}}
\end{align*}
\end{Proposition}
\begin{proof}
The first equality follows immediately from the adjoint relationship, and the second can be proven as a direct computation using the well-known formula for the norm of the Macdonald polynomials.
\end{proof}

\subsection{}

We will use the following properties of the vertex operators:
\begin{Proposition}\label{commutation} As operators on the completion of $\mathcal{F}$, we have:
\begin{enumerate}[label=(\alph*)]
  \item\label{pt1} $$\Gamma_{-}(z) M_{\emptyset} =  M_{\emptyset}$$
  where $\emptyset$ denotes the empty partition.
    \item\label{pt2} $$z^L\Gamma_{\pm}(w) = \Gamma_{\pm}(wz) z^L$$
      \item\label{pt3} $$\Gamma_-(z) \Gamma_+(w) = \prod_{i=0}^\infty \frac{1-\hbar \frac{w}{z} q^i}{1-\frac{w}{z}q^i} \Gamma_+(w) \Gamma_-(z)$$
      (Note that the product converges if $|q|<1$).
\end{enumerate}
\end{Proposition}
\begin{proof}
Parts \ref{pt1} and \ref{pt2} follow from (\ref{pieriplus}) and (\ref{pieriminus}), along with the definition of $z^L$. Part \ref{pt3} follows from the definitions of $\Gamma_{\pm}(z)$.
\end{proof}

\subsection{Proof of Theorem \ref{mthm}, Step 1}
Fix a partition $\lambda$. In this subsection, we identify the product in (\ref{prform}) as a certain matrix element of an operator on the completion of $\mathcal{F}$.
Let 
\begin{align*}
  \Lambda_i:= \widehat{z}_{i}^L \Gamma_{\tau_{\lambda}(i)}(1) 
\end{align*}
where $\tau_{\lambda}(i)\in \{+,-\}$ is defined by
\begin{align*}
    \tau_{\lambda}(i)&:=
    \begin{cases}
    + &  \text{if} \ \ i \geq 0  \ \ \text{and} \ \ \textsf{v}_i-\textsf{v}_{i+1}=1 \\ + & \text{if} \ \ i<0 \ \ \text{and} \ \ \textsf{v}_i-\textsf{v}_{i+1}=0 \\
    - & \text{otherwise}
    \end{cases} 
\end{align*}
where $\textsf{v}$ is as in Section 1. By definition, $\widehat{z}_i^L$ is the operator that acts as $\widehat{z}_i^L\cdot M_{\mu}= \left(\widehat{z}_i\right)^{|\mu|} M_{\mu}$ where $\widehat{z}_i$ is as in Section 1.4. It follows that $\Lambda_i$ is an operator on the completion of $\mathcal{F}$.

The function $\tau_\lambda$ can be thought of as keeping track of the slopes along the top of the partition, see Figure \ref{yng2}.

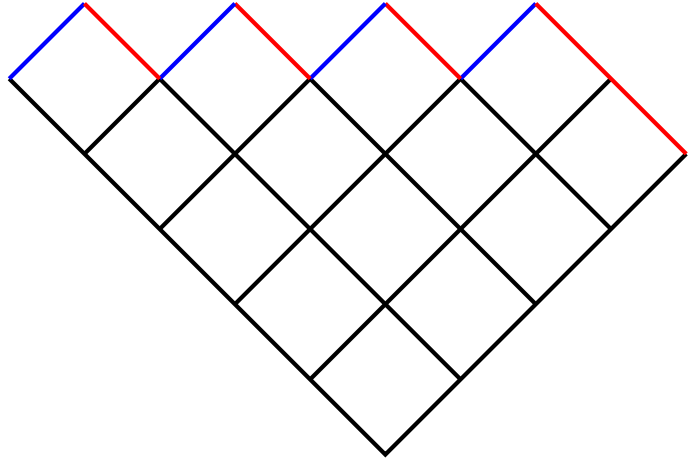
\begin{figure}[ht]
\centering
\begin{tikzpicture}[roundnode/.style={circle, draw=black, very thick, minimum size=5mm},squarednode/.style={rectangle, draw=black, very thick, minimum size=5mm}] 
\draw[ultra thick] (-5,5)--(0,0) -- (4,4);
\draw[ultra thick] (-3,5)--(1,1);
\draw[ultra thick,blue] (-5,5)--(-4,6);
\draw[ultra thick] (-1,5)--(2,2);
\draw[ultra thick] (-4,4)--(-3,5);
\draw[ultra thick]  (1,5)--(3,3);
\draw[ultra thick] (-1,1)--(3,5);
\draw[ultra thick] (-2,2)--(1,5);
\draw[ultra thick,blue] (1,5)--(2,6);
\draw[ultra thick] (-3,3)--(-1,5);
\draw[ultra thick, red] (4,4)--(2,6);
\draw[ultra thick,red] (1,5)--(0,6);
\draw[ultra thick,blue] (-1,5)--(0,6);
\draw[ultra thick,red] (-1,5)--(-2,6);
\draw[ultra thick,blue] (-3,5)--(-2,6);
\draw[ultra thick,red] (-3,5)--(-4,6);
\end{tikzpicture}
\caption{The upper boundary of $\lambda$ is a piecewise linear function. The slope of the segment from column $i$ to column $i+1$ is $-\tau_{\lambda}(i)$. In this case, $\lambda=(5,4,3,2)$ and $\textsf{v}=(\textsf{v}_i)=(\ldots,0,0,1,1,2,2,3,2,2,1,0,0,\ldots)$. Then $-\tau_{\lambda}(i)$ is $+$ if the boundary segment to the right of column $i$ is blue, and is $-$ otherwise.} \label{yng2}
\end{figure}

For definiteness, let $-r$ and $s$ be the indices of the leftmost and rightmost nonzero entries of $\textsf{v}$. We have the following:
\begin{Proposition} 
\begin{equation*}
     \left\langle M_{\emptyset} \middle| \Gamma_{-}(1) \Lambda_{-r} \Lambda_{-r+1} \ldots \Lambda_{s}  \middle| M_{\emptyset} \right\rangle = \prod_{\square \in \lambda} \prod_{i=0}^\infty \frac{1-\hbar z_{\square} q^i}{1-z_{\square} q^i}
\end{equation*}
\end{Proposition}
\begin{proof}
We first commute all of the $\widehat{z}_i^L$ operators to the right, using Proposition \ref{commutation}, and use the fact that $\widehat{z}_i^L \cdot M_{\emptyset}= M_{\emptyset}$:
\begin{align*}
      \left\langle M_{\emptyset} \middle| \Gamma_{-}(1) \Lambda_{-r} \Lambda_{-r+1} \ldots \Lambda_{s}  \middle| M_{\emptyset} \right\rangle =   \left\langle M_{\emptyset} \middle| \Gamma_{-}(1) \prod_{i=-r}^s  \Gamma_{\tau_{\lambda}(i)}\left(  \widehat{z}_{-r} \ldots \widehat{z}_i \right) \middle| M_{\emptyset} \right\rangle
\end{align*}
with the product ordered appropriately. Next, we commute each $\Gamma_{-}(w)$ to the right using Proposition \ref{commutation}, starting with the rightmost one. Suppose that $i_0 < i_1$ are such that $\tau_{\lambda}(i_0)=-$ and $\tau_{\lambda}(i_1)=+$. Then commuting $\Gamma_{\tau_{\lambda}(i_0)}(\widehat{z}_{-r} \ldots \widehat{z}_{i_0})$ across $\Gamma_{\tau_{\lambda}(i_1)}(\widehat{z}_{-r} \ldots \widehat{z}_{i_1})$ introduces a factor of 
\begin{align*}
\prod_{i=0}^\infty \frac{1-\hbar  \widehat{z}_{i_0+1} \ldots \widehat{z}_{i_1} q^i}{1- \widehat{z}_{i_0+1} \ldots \widehat{z}_{i_1} q^i} = \prod_{i=0}^\infty \frac{1-\hbar z_{\square} q^i}{1-z_{\square} q^i}
\end{align*}
where $\square$ is the box in row $a$ and column $b$ of $\lambda$ (in the standard coordinates) for
\begin{align*}
    a&=|\{i\geq i_1 \mid \tau(i)=1\}| \\
    b&=|\{i\leq i_0 \mid \tau(i)=-1\}|
\end{align*}
Commuting all such terms, we obtain
\begin{align*}
 & \left\langle M_{\emptyset} \middle| \Gamma_{-}(1) \Lambda_{-r} \Lambda_{-r+1} \ldots \Lambda_{s}  \middle| M_{\emptyset} \right\rangle \\
  &= \prod_{\square \in \lambda} \prod_{i=0}^\infty \frac{1-\hbar z_{\square} q^i}{1-z_{\square} q^i} \left\langle M_{\emptyset} \Bigg|  \prod_{\substack{i=-r \\ \tau_{\lambda}(i)=+}}^s \Gamma_{\tau_{\lambda}(i)} \left(  \widehat{z}_{-r} \ldots \widehat{z}_i  \right)  \Bigg| M_{\emptyset}\right\rangle \\
  &= \prod_{\square \in \lambda} \prod_{i=0}^\infty \frac{1-\hbar z_{\square} q^i}{1-z_{\square} q^i}
\end{align*}
\end{proof}

\subsection{}
\begin{Definition}\label{shape}
We say that an $(r+s+1)$-tuple of partitions $(\lambda^{-r}, \ldots , \lambda^s)$ \textit{interlace according to the shape of a partition} $\lambda$ if
\begin{itemize}
\item $\tau_{\lambda}(i)=+ \implies \lambda^{i} \succ \lambda^{i+1}$ 
\item $\tau_{\lambda}(i)=- \implies \lambda^{i} \prec \lambda^{i+1}$ 
\item $l(\lambda^{i}) \leq \mathsf{v}_i$
\end{itemize} 
We write $S_{\lambda}$ for the set of all $(r+s+1)$-tuples of partitions that interlace according to the shape of $\lambda$.
\end{Definition}
\begin{Proposition} We can rewrite the vertex function (\ref{vf1}) as:
\begin{align}\label{vf2}
   \mathcal{Z}_{\lambda} = \sum_{(\lambda^{-r}, \ldots, \lambda^{s}) \in S_{\lambda}} \left( \alpha_{\lambda^0} \prod_{i=-r}^{s-1} \beta_{\lambda^i,\lambda^{i+1}} \prod_{i=-r}^s \gamma_{\lambda^i} z_i^{|\lambda^i|} \right) 
\end{align}
where 
\begin{align*}
    \alpha_{\lambda^0} &= \prod_{j=1}^{l(\lambda^0)} \frac{(\hbar^{l(\lambda^0)-j+1})_{\lambda^0_{j}}}{(q \hbar^{l(\lambda^0)-j})_{\lambda^0_{j}}} \\
    \beta_{\lambda^i,\lambda^{i+1}}  &= \begin{cases} \prod\limits_{j=1}^{\mathsf{v}_i} \prod\limits_{k=1}^{\mathsf{v}_{i+1}} \frac{(\hbar^{j-k+1-\sigma(i+1)})_{\lambda^{i+1}_k-\lambda^i_j}}{(q\hbar^{j-k-\sigma(i+1)})_{\lambda^{i+1}_k-\lambda^i_j}} & i \geq 0 \\ \prod\limits_{j=1}^{\mathsf{v}_i} \prod\limits_{k=1}^{\mathsf{v}_{i+1}} \frac{(\hbar^{j-k-\sigma(i+1)})_{\lambda^{i+1}_k-\lambda^i_j}}{(q\hbar^{j-k-1-\sigma(i+1)})_{\lambda^{i+1}_k-\lambda^i_j}} & i < 0 \end{cases} \\
    \gamma_{\lambda^i} &= \prod_{j,k=1}^{\mathsf{v}_i} \frac{(q \hbar^{j-k})_{\lambda^i_k-\lambda^i_j}}{( \hbar^{j-k+1})_{\lambda^i_k-\lambda^i_j}}
\end{align*}
and $\sigma(i):=\textsf{v}_{i-1}-\textsf{v}_i$.
\end{Proposition}
\vspace{-.43cm}
\begin{proof}
Observe that
\begin{equation*}
    i<0 \ \ \text{and} \ \ k=j \implies (\hbar^{k-j})_{d_{i+1,k}-d_{i,j}} = (1)_{d_{i+1,k}-d_{i,j}}
\end{equation*}
which is nonzero only if $d_{i+1,k}-d_{i,j}\leq 0$. Similarly,
\begin{equation*}
i<0 \ \ \text{and} \ \ k=j+1 \implies \frac{1}{(q\hbar^{k-j-1})_{d_{i+1,k}-d_{i,j}}} = \frac{1}{(q)_{d_{i+1,k}-d_{i,j}}}
\end{equation*}
which is nonzero only if $d_{i+1,k}-d_{i,j}\geq 0$. An examination of the terms in (\ref{vf1}) gives a similar condition for $d_{i,j}$ where $i\geq 0$. This condition has the following interpretation: starting at the corner box of $\lambda$, the indices $d_{i,j}$ of a nonzero term in (\ref{vf1}) must be weakly increasing whenever we move up and to the left or up and to the right.

In particular, we have $d_{i,1}\leq d_{i,2} \leq \ldots d_{i,p_i}$ for all $i$. Letting 
\begin{equation*}
    \lambda^{i}=(d_{i,\mathsf{v}_i}, \ldots , d_{i,1})
\end{equation*}
the comments in the preceding paragraph show that the sum for the vertex function is indexed $S_\lambda$. Re-indexing the vertex function in this way gives (\ref{vf2}).
\end{proof}

\subsection{}
\begin{Lemma} For any fixed $(\lambda^{-r}, \ldots, \lambda^s)\in S_{\lambda}$, we have
\begin{align*}
    \gamma_{\lambda^i} &= \prod_{1 \leq k < j \leq \mathsf{v}_i}\frac{(q \hbar^{j-k})_{\lambda^i_k- \lambda^{i}_j}}{( \hbar^{j-k+1})_{\lambda^i_k- \lambda^{i}_j}} \prod_{1 \leq j < k \leq \mathsf{v}_i} \left( \frac{\hbar}{q}\right)^{\lambda^i_j-\lambda^i_k}\frac{(q \hbar^{k-j-1})_{\lambda^i_j- \lambda^{i}_k}}{( \hbar^{k-j})_{\lambda^i_j- \lambda^{i}_k}} \\
    \end{align*}
    If $i\geq 0$, then
    \begin{align*}
    \beta_{\lambda^i,\lambda^{i+1}} &= \prod_{j=1}^{\mathsf{v}_i} \prod_{\substack{k=1 \\ k <_{-\tau(i)} j}}^{\mathsf{v}_{i+1}}  \frac{(\hbar^{j-k+1-\sigma(i+1)})_{\lambda^{i+1}_k-\lambda^i_j}}{(q\hbar^{j-k-\sigma(i+1)})_{\lambda^{i+1}_k-\lambda^i_j}}
    \prod_{j=1}^{\mathsf{v}_i}    \prod_{\substack{k=1 \\ j <_{\tau(i)} k }}^{\mathsf{v}_{i+1}} \left(\frac{q}{\hbar} \right)^{\lambda^{i}_j-\lambda^{i+1}_k} \frac{(\hbar^{k-j+\sigma(i+1)})_{\lambda^{i}_j-\lambda^{i+1}_k}}{(q \hbar^{k-j-1+\sigma(i+1)})_{\lambda^{i}_j-\lambda^{i+1}_k}} 
    \end{align*}
   where  $<_{-}=<$ and $<_{+}=\leq$ and we have abbreviated $\tau$ for $\tau_{\lambda}$. A similar formula holds for $i<0$, but with a shift by $\hbar$ or $\hbar^{-1}$ in the appropriate Pochhammer symbols.
   \end{Lemma}

   \begin{proof}
   This follows from the identity
   \begin{equation*}
    \frac{(\hbar x)_{-n}}{(qx)_{-n}}=\left(\frac{q}{\hbar}\right)^n \frac{(\frac{1}{x})_n}{(\frac{q}{\hbar}\frac{1}{x})_n}
\end{equation*}
   \end{proof}
   

For convenience, we further introduce
\begin{align*}
    \delta_{\lambda^i} = \prod_{1 \leq k < j \leq \mathsf{v}_i}\frac{(q \hbar^{j-k})_{\lambda^i_k- \lambda^{i}_j}}{( \hbar^{j-k+1})_{\lambda^i_k- \lambda^{i}_j}} \ \ \ \text{and} \ \ \ \epsilon_{\lambda^i} =  \prod_{1 \leq j \leq k \leq \mathsf{v}_i} \frac{(q \hbar^{k-j-1})_{\lambda^i_j- \lambda^{i}_k}}{( \hbar^{k-j})_{\lambda^i_j- \lambda^{i}_k}}
\end{align*}
so that
\begin{equation*}
    \gamma_{\lambda^i} = \delta_{\lambda^i} \epsilon_{\lambda^i} \prod_{1 \leq j \leq k \leq \mathsf{v}_i} \left( \frac{\hbar}{q}\right)^{\lambda^i_j-\lambda^i_k}
\end{equation*}

\subsection{Proof of Theorem \ref{mthm}, Step 2}

We are now ready to finish the proof of Theorem \ref{mthm} by examining step by step the result of applying each $\Lambda_i$ to $M_{\emptyset}$. Formulas (\ref{pieriplus}) and (\ref{pieriminus}) show that terms of 
\begin{equation*}
 \Gamma_{-}(1) \Lambda_{-r} \Lambda_{-r+1} \ldots \Lambda_{s}   M_{\emptyset} 
\end{equation*}
are indexed by $(r+s+1)$-tuples of partitions that satisfy the first two conditions of Definition \ref{shape}, but a-priori without satisfying the length requirement. However, taking the $M_{\emptyset}$-coefficient of $\Gamma_{-}(1) \Lambda_{-r} \Lambda_{-r+1} \ldots \Lambda_{s}  M_{\emptyset}$ ensures that the only terms that contribute actually arise from elements of $S_{\lambda}$, a fact that we assume in the remaining computations below.

Now, let $(\lambda^{-r}, \ldots , \lambda^{s}) \in S_{\lambda}$. Let us also fix $i \geq 0$ so that $\tau(i)=+$, where as before we write $\tau$ for $\tau_{\lambda}$.
Then by (\ref{pieriplus}), we have
\begin{align*}
   M_{\lambda^i}  \ \text{coefficient of} \ \Lambda_i\left( M_{\lambda^{i+1}} \right) = \widehat{z}_i^{|\lambda_i|}  \prod_{1\leq j \leq k \leq \mathsf{v}_i} \left( \frac{(\hbar^{k-j+1})_{\lambda^i_j-\lambda^{i+1}_k}}{(q\hbar^{k-j})_{\lambda^i_j-\lambda^{i+1}_k}} \frac{(q\hbar^{k-j})_{\lambda^i_j-\lambda^i_k}}{(\hbar^{k-j+1})_{\lambda^i_j-\lambda^i_k}} \right)  \\
   \prod_{1\leq k<j \leq \mathsf{v}_i} \left( \frac{(\hbar^{j-k})_{\lambda^{i+1}_k-\lambda^i_j}}{(q\hbar^{j-k-1})_{\lambda^{i+1}_k-\lambda^i_j}} \frac{(q \hbar^{j-k-1})_{\lambda^{i+1}_k-\lambda^{i+1}_j}}{(\hbar^{j-k})_{\lambda^{i+1}_k-\lambda^{i+1}_j}}  \right) 
\end{align*}
where we adopt the convention that $\lambda^{i+1}_j=0$ for $j>\mathsf{v}_{i+1}$. Separating such terms, we have
\begin{align*}
   M_{\lambda^i}  \ \text{coefficient of} \ \Lambda_i\left( M_{\lambda^{i+1}} \right)  = \widehat{z}_i^{|\lambda_i|}  \prod_{j=1}^{\mathsf{v}_i} \frac{(\hbar^{\mathsf{v}_i-j+1})_{\lambda^i_j}}{(q\hbar^{\mathsf{v}_i-j})_{\lambda^i_j}} \prod_{k=1}^{\mathsf{v}_{i+1}}  \frac{(q \hbar^{\mathsf{v}_i-k-1})_{\lambda^{i+1}_k}}{(\hbar^{\mathsf{v}_i-k})_{\lambda^{i+1}_k}} \\
   \prod_{j=1}^{\mathsf{v}_i} \prod_{\substack{k=1 \\ j \leq k}}^{\mathsf{v}_{i+1}} \frac{(\hbar^{k-j+1})_{\lambda^i_j-\lambda^{i+1}_k}}{(q\hbar^{k-j})_{\lambda^i_j-\lambda^{i+1}_k}} 
     \prod_{j=1}^{\mathsf{v}_i} \prod_{\substack{k=1 \\ k<j}}^{\mathsf{v}_{i+1}} \frac{(\hbar^{j-k})_{\lambda^{i+1}_k-\lambda^i_j}}{(q\hbar^{j-k-1})_{\lambda^{i+1}_k-\lambda^i_j}}  \prod_{1\leq j \leq k \leq \mathsf{v}_i}  \frac{(q\hbar^{k-j})_{\lambda^i_j-\lambda^i_k}}{(\hbar^{k-j+1})_{\lambda^i_j-\lambda^i_k}}
     \\
     \prod_{1\leq k<j \leq \mathsf{v}_{i+1}} \frac{(q \hbar^{j-k-1})_{\lambda^{i+1}_k-\lambda^{i+1}_j}}{(\hbar^{j-k})_{\lambda^{i+1}_k-\lambda^{i+1}_j}}   
\end{align*}
We also have
\begin{align*}
    \beta_{\lambda^i, \lambda^{i+1}} = \prod_{j=1}^{\mathsf{v}_i} \prod_{\substack{k=1 \\ k < j}}^{\mathsf{v}_{i+1}}  \frac{(\hbar^{j-k})_{\lambda^{i+1}_k-\lambda^i_j}}{(q\hbar^{j-k-1})_{\lambda^{i+1}_k-\lambda^i_j}}
    \prod_{j=1}^{\mathsf{v}_i}    \prod_{\substack{k=1 \\ j \leq k }}^{\mathsf{v}_{i+1}} \left(\frac{q}{\hbar} \right)^{\lambda^{i}_j-\lambda^{i+1}_k} \frac{(\hbar^{k-j+1})_{\lambda^{i}_j-\lambda^{i+1}_k}}{(q \hbar^{k-j})_{\lambda^{i}_j-\lambda^{i+1}_k}}
\end{align*}
since $\sigma(i+1)=1$. So we have that
\begin{align*}
      M_{\lambda^i}  \ \text{coefficient of} \ \Lambda_i\left( M_{\lambda^{i+1}} \right)  = \widehat{z}_i^{|\lambda_i|}  \left( \beta_{\lambda^i, \lambda^{i+1}} \right) \left( \delta_{\lambda^i} \right) \left( \epsilon_{\lambda^{i+1}}  \right) \prod_{j=1}^{\mathsf{v}_i} \frac{(\hbar^{\mathsf{v}_i-j+1})_{\lambda^i_j}}{(q\hbar^{\mathsf{v}_i-j})_{\lambda^i_j}} \\
      \prod_{k=1}^{\mathsf{v}_{i+1}}  \frac{(q \hbar^{\mathsf{v}_i-k-1})_{\lambda^{i+1}_k}}{(\hbar^{\mathsf{v}_i-k})_{\lambda^{i+1}_k}}     \prod_{j=1}^{\mathsf{v}_i}  \prod_{\substack{k=1 \\ j <_{\tau(i)} k }}^{\mathsf{v}_{i+1}} \left(\frac{\hbar}{q} \right)^{\lambda^{i}_j-\lambda^{i+1}_k}
\end{align*}
An identical calculation using (\ref{pieriminus}) shows that the same formula also holds if $\tau(i)=-$. So applying up to $\Gamma_{\tau(0)}$, the first two products in the previous expression will all cancel except for the $p_0$ term, leaving
\begin{align*}
 &  \alpha_{\lambda^0} \left( \prod_{i=0}^{s-1} \beta_{\lambda^{i}, \lambda^{i+1}} \epsilon_{\lambda^{i+1}} \prod_{j=1}^{\mathsf{v}_i}  \prod_{\substack{k=1 \\ j <_{\tau(i)} k }}^{\mathsf{v}_{i+1}} \left(\frac{\hbar}{q} \right)^{\lambda^{i}_j-\lambda^{i+1}_k} \right) \left( \prod_{i=0}^s \delta_{\lambda^i} \widehat{z}_i^{|\lambda_i|} \right)   \\
 =&     \alpha_{\lambda^0} \delta_{\lambda^0} \left( \prod_{i=0}^{s-1} \beta_{\lambda^{i}, \lambda^{i+1}}   \prod_{j=1}^{\mathsf{v}_i}  \prod_{\substack{k=1 \\ j <_{\tau(i)} k }}^{\mathsf{v}_{i+1}} \left(\frac{\hbar}{q} \right)^{\lambda^{i}_j-\lambda^{i+1}_k} \right)  \left( \prod_{i=1}^s \gamma_{\lambda^i}
 \prod_{1 \leq j < k \leq \mathsf{v}_i} \left( \frac{\hbar}{q}\right)^{\lambda^i_k-\lambda^i_j} \right)    \prod_{i=0}^s  \widehat{z}_i^{|\lambda_i|}
\end{align*}

Now, suppose that $i<0$. We will perform a similar calculation here. Again, suppose that $\tau(i)=+$. In this case, $\mathsf{v}_i=\mathsf{v}_{i+1}$ and we have
\begin{align*}
      M_{\lambda^i}  \ \text{coefficient of} \ \Lambda_i\left( M_{\lambda^{i+1}} \right)  = \widehat{z}_i^{|\lambda^i|} \prod_{1\leq j \leq k \leq \mathsf{v}_i} \left( \frac{(\hbar^{k-j+1})_{\lambda^i_j-\lambda^{i+1}_k}}{(q\hbar^{k-j})_{\lambda^i_j-\lambda^{i+1}_k}} \frac{(q\hbar^{k-j})_{\lambda^i_j-\lambda^i_k}}{(\hbar^{k-j+1})_{\lambda^i_j-\lambda^i_k}} \right) \\
     \prod_{1\leq k<j \leq \mathsf{v}_i} \left( \frac{(\hbar^{j-k})_{\lambda^{i+1}_k-\lambda^i_j}}{(q\hbar^{j-k-1})_{\lambda^{i+1}_k-\lambda^i_j}} \frac{(q \hbar^{j-k-1})_{\lambda^{i+1}_k-\lambda^{i+1}_j}}{(\hbar^{j-k})_{\lambda^{i+1}_k-\lambda^{i+1}_j}} \right)
\end{align*}
We also have
\begin{align*}
    \beta_{\lambda^i,\lambda^{i+1}} = \prod_{1\leq k < j \leq \mathsf{v}_i} \frac{(\hbar^{j-k})_{\lambda^{i+1}_k-\lambda^i_j}}{(q\hbar^{j-k-1})_{\lambda^{i+1}_k-\lambda^i_j}}
   \prod_{1\leq j \leq k \leq \mathsf{v}_i} \left(\frac{q}{\hbar} \right)^{\lambda^{i}_j-\lambda^{i+1}_k} \frac{(\hbar^{k-j+1})_{\lambda^{i}_j-\lambda^{i+1}_k}}{(q \hbar^{k-j})_{\lambda^{i}_j-\lambda^{i+1}_k}}
\end{align*}
since $\sigma(i+1)=0$.
So we have that 
\begin{align*}
      M_{\lambda^i}  \ \text{coefficient of} \ \Lambda_i\left( M_{\lambda^{i+1}} \right)  =  \widehat{z}_i^{|\lambda^i|}  \left( \beta_{\lambda^i,\lambda^{i+1}}\right) \left( \delta_{\lambda^i} \right) \left( \epsilon_{\lambda^{i+1}}  \right) \prod_{1\leq j \leq k \leq \mathsf{v}_i} \left(\frac{\hbar}{q} \right)^{\lambda^{i}_j-\lambda^{i+1}_k} 
\end{align*}
A similar calculation applies for $i<0$ and $\tau(i)=-$. Applying the rest of the operators, we find that the term of  $ \left\langle M_{\emptyset} \middle| \Gamma_{-}(1) \Lambda_{-r} \Lambda_{-r+1} \ldots \Lambda_{s}  \middle| M_{\emptyset} \right\rangle$ corresponding to $(\lambda^{-r}, \ldots, \lambda^{s})$ is 
\begin{align}
      \alpha_{\lambda^0} \left( \prod_{i=-r}^{s-1} \beta_{\lambda^{i}, \lambda^{i+1}} \prod_{j=1}^{\mathsf{v}_i}  \prod_{\substack{k=1 \\ j <_{\tau(i)} k }}^{\mathsf{v}_{i+1}} \left(\frac{\hbar}{q} \right)^{\lambda^{i}_j-\lambda^{i+1}_k} \right) \left( \prod_{i=-r}^s \gamma_{\lambda^i}  \prod_{1 \leq j < k \leq \mathsf{v}_i} \left( \frac{\hbar}{q}\right)^{\lambda^i_k-\lambda^i_j} \right)  \prod_{i=-r}^s \widehat{z}_i^{|\lambda^i|} 
\end{align}
Comparing this with (\ref{vf2}) and recalling the definition of $\widehat{z}_{i}$, we see that $\mathcal{Z}_{\lambda}$ and $ \left\langle M_{\emptyset} \middle| \Gamma_{-}(1) \Lambda_{-r} \Lambda_{-r+1} \ldots \Lambda_{s}  \middle| M_{\emptyset} \right\rangle$ agree term-wise, up to a multiple of $\frac{\hbar}{q}$. So the following lemma completes the proof of Theorem \ref{mthm}.
\begin{Lemma} Let $\lambda$ be a partition and let $(\lambda^{-r}, \ldots, \lambda^{s})\in S_{\lambda}$. Then
\begin{align*}
 \sum_{i=-r}^{s-1} \sum_{j=1}^{\mathsf{v}_i} \sum_{\substack{k=1 \\ j <_{\tau(i)} k }}^{\mathsf{v}_{i+1}} \lambda^i_j - \lambda^{i+1}_k + \sum_{i=-r}^s \sum_{1 \leq j < k \leq \mathsf{v}_i} \lambda^i_k - \lambda^i_j = -\sum_{i=-r}^s  \widehat{\sigma}_{\lambda}(i) |\lambda^i| 
 \end{align*}
\end{Lemma}
\begin{proof}
We proceed by induction on $|\lambda|$. Call the left hand side $C_{\lambda}$. Without loss of generality, suppose that after removing a box with content $i_0>0$, we are still left with a partition, which we call $\mu$. This box corresponds to $\lambda^{i_0}_{1}$ in the summation for $C_{\lambda}$. Then we have
\begin{itemize}
    \item $\widehat{\sigma}_{\lambda}(i_0) =0 $ 
    \item $\widehat{\sigma}_{\lambda}(i_0+1) =1$  
    \item $\widehat{\sigma}_{\mu}(i_0)=1 $
    \item $\widehat{\sigma}_{\mu}(i_0+1)=0 $
    \item $i\notin \{i_0,i_0+1\} \implies \widehat{\sigma}_{\mu}(i)=\widehat{\sigma}_{\lambda}(i) \ \  \text{and} \ \ |\mu^i|=|\lambda^i|$
\end{itemize}
Separating the terms in $C_{\lambda}$ that depend on $\lambda^{i_0}_{1}$ and using the induction hypothesis, we obtain
\begin{align*}
   C_{\lambda}= C_{\mu}+ \sum_{j=1}^{\mathsf{v}_{i_0}-1} \left( \lambda^{i_0}_1 - \lambda^{i_0+1}_{j} \right) + \sum_{j=1}^{\mathsf{v}_{i_0}} \left( \lambda^{i_0}_j- \lambda^{i_0}_1  \right) \\
   = C_{\mu}-|\lambda^{i_0+1}|+|\lambda^{i_0}| - \lambda^{i_0}_1 \\ 
   =  -\sum_{i=-r}^s \widehat{\sigma}_{\mu}(i) \left(|\lambda^i|-\delta_{i,i_0}\lambda^{i_0}_1\right)  -|\lambda^{i_0+1}|+|\lambda^{i_0}| -\lambda^{i_0}_1 \\
   = -\sum_{\substack{i=-r \\ i\notin \{i_0,i_0+1\}}}^s \widehat{\sigma}_{\lambda} (i) |\lambda^i| -\left(|\lambda^{i_0}|-\lambda^{i_0}_1 \right) -|\lambda^{i_0+1}|+|\lambda^{i_0}| - \lambda^{i_0}_1 \\
   =  -\sum_{\substack{i=-r \\ i\notin \{i_0,i_0+1\}}}^s \widehat{\sigma}_{\lambda} (i) |\lambda^i| -|\lambda^{i_0+1}|\\
   = -\sum_{i=-r}^s \widehat{\sigma}_{\lambda}(i) |\lambda^i|
\end{align*}
as desired.
\end{proof}

\section{Vertex functions of $A_{n}$-Quiver Varieties and 3d Mirror Symmetry}
In the final section, we describe an application of Theorem \ref{mthm} to finite and affine type $A_n$-quiver varieties.

\subsection{}
We consider quivers of type $A_n$, i.e. with vertex set $I=\{0, 1, \ldots, n-1\}$, and arrows $i\to i+1$. Let $\mathsf{v}=(\mathsf{v}_0,\ldots , \mathsf{v}_{n-1})$ and $\mathsf{w}=(\mathsf{w}_0, \ldots, \mathsf{w}_{n-1})$ be the dimension and framing dimension vectors, respectively. This information defines a Nakajima quiver variety $X={\cal{M}}(\textsf{v}, \textsf{w})$ as in Section 2.1. Our choice of stability condition is given by the $\prod_{i=0}^{n-1} GL(\mathsf{v}_i)$-character
$$
\chi:  (g_i) \mapsto \prod_{i=0}^{n-1} \det(g_i)
$$

The torus
$$\bT = \mathbb{C}^\times_\hbar \times \bA = \mathbb{C}^\times_\hbar\times \left( \mathbb{C}^\times\right)^{w_0+\ldots + w_{n-1}}$$
acts on $X$, where $\bA$ preserves the symplectic form and $\mathbb{C}^\times_\hbar$ scales it with character $\hbar$.

A $\bT$-fixed point on $X$ is naturally indexed by the set of $(\mathsf{w}_0+ \ldots + \mathsf{w}_{n-1})$-tuples of $n$-colored partitions, such that the total number of boxes of color $i$ is $\mathsf{v}_i$. This is because our choice of stability implies that the images of the framings generate the vector space $\bigoplus_{i=0}^{n-1} V_{i}$ under the action of the arrows of the doubled $A_n$ quiver. The moment map condition implies that the subspace generated by one particular framing dimension takes the shape of a Young diagram, and the boxes in a partition correspond to linearly independent directions. We think of these partitions as being based at the vertex corresponding to the framing, see Figure \ref{yng3} below. With this understood, the boxes are ``colored" by the vector space that they lie over, see \cite{neg} Section 2.2 for more details.

Let ${\bf V}_p(\aa,\zz)\in K_{\bT}(X)_{loc}[[\zz]]$ be the restriction of vertex function for $X$, defined in Section 7 of \cite{pcmilect}, to a fixed point $p\in X^{\bT}$. As explained in Section 2.2, ${\bf V}_p(\aa,\zz)$ can be computed explicitly.

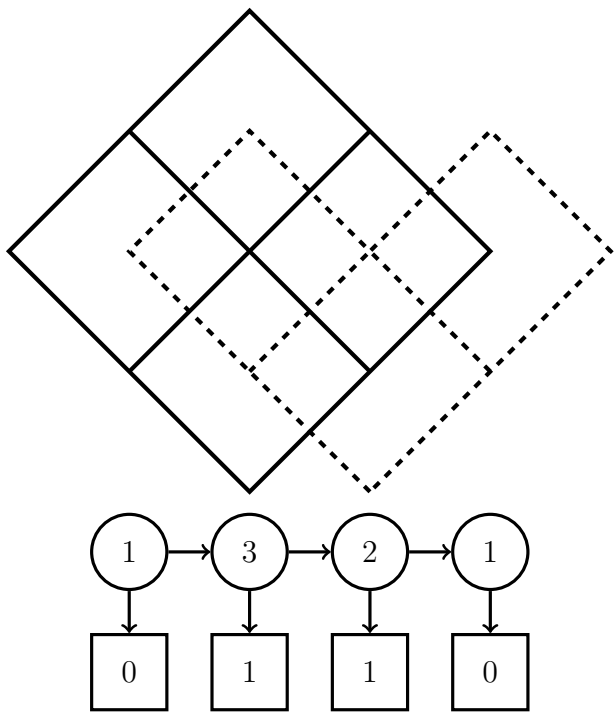
\begin{figure}[!ht]
\centering
\begin{tikzpicture}[roundnode/.style={circle, draw=black, very thick, minimum size=10mm},squarednode/.style={rectangle, draw=black, very thick, minimum size=10mm},scale=.8] 

\node[roundnode](3) at (0,0){1};
\node[roundnode](2) at (-2,0){2};
\node[roundnode](1) at (-4,0){3};
\node[roundnode](0) at (-6,0){1};

\node[squarednode](f3) at (0,-2){0};
\node[squarednode](f2) at (-2,-2){1};
\node[squarednode](f1) at (-4,-2){1};
\node[squarednode](f0) at (-6,-2){0};

\draw[very thick,->] (0)--(1);
\draw[very thick,->] (1)--(2);
\draw[very thick,->] (2)--(3);

\draw[very thick,->] (0)--(f0);
\draw[very thick,->] (1)--(f1);
\draw[very thick,->] (2)--(f2);
\draw[very thick,->] (3)--(f3);

\draw[ultra thick, dashed] (-2,1)--(-6,5)--(-4,7)--(-2,5)--(0,7)--(2,5)--cycle;
\draw[ultra thick, dashed] (0,3)--(-2,5)--(-4,3) ;

\draw[ultra thick] (-4,1)--(-8,5)--(-4,9)--(0,5)--cycle;
\draw[ultra thick] (-6,3)--(-2,7);
\draw[ultra thick] (-2,3)--(-6,7);

\end{tikzpicture}
\caption{A fixed point on the $A_4$ quiver variety determined by dimension $\textsf{v}=(1,3,2,1)$ and framing dimension $\textsf{w}=(0,1,1,0)$. The partitions $(2,2)$ (solid lines) and $(2,1)$ (dashed lines) index a fixed point on the quiver variety, and correspond to the framings at vertices 1 and 2, respectively. Observe that the total number of boxes above each vertex is the dimension corresponding to that vertex. This is precisely the coloring condition.} \label{yng3}
\end{figure}

\subsection{}
The real Lie algebra $\Lie_{\matR}(\bA)$ is naturally equipped with a set of hyperplanes $\{ \alpha^{\perp} \}$, where $\alpha$ runs over the set of $\bA$-characters appearing as $\bA$-weights of the tangent spaces $T_{p} X$ for $p\in X^{\bT}$. 
The complement of these hyperplanes is a union of connected components 
\begin{equation} \label{chamb}
\Lie_{\matR}(\bA)\setminus \{ \alpha^{\perp}\}= \bigcup \mathfrak{C}
\end{equation}
which are called {\it chambers}. Choosing a chamber $\mathfrak{C}$ is equivalent to choosing a cocharacter $f: \matC^\times \to \bA$ disjoint from the hyperplanes $\{\alpha^{\perp}\}$.

\subsection{}
Fix a choice of chamber $\mathfrak{C}\subset \Lie_{\matR}(\bA)$, with corresponding co-character $f: \mathbb{C}^\times \to \bA $. For $p \in X^\bT$, we define the limit of the vertex function with respect to $\mathfrak{C}$ to be
\begin{equation} \label{vlims}
    \textbf{V}_{p}(0_\mathfrak{C}, \bs{z}) :=\lim_{w\to 0} \textbf{V}_{p} (f(w),\bs{z}) \in \matQ(\hbar,q)[[z]]
\end{equation}
We recall that the vertex functions $\textbf{V}_{p}(\bs{a}, \bs{z})$ are balanced in equivariant parameters see Section 6.1 \cite{pcmilect}, which means the coefficients of vertex functions depend on the $\bA$-weights $w$ though a combinations of the form
$$
\dfrac{1-w \alpha}{1- w \beta}
$$
where $\alpha$ and $\beta$ stand for some monomials in $\hbar$ and $q$ (it also easy to see this directly from (\ref{vertild})). Note that this expression has a well defined limits at $w\to0$ and $w\to \infty$. Thus,  (\ref{vlims}) is well defined for any chamber~$\fC$.

As remarked earlier, the fixed point $p$ corresponds to a $(\textsf{w}_0 + \ldots + \textsf{w}_{n-1})$-tuple of partitions. We write the set of these partitions as $\Omega_p$, and coordinates on the torus $\bA$ are thus given by $(a_\lambda)_{\lambda \in \Omega_p}$. So we write
$$
  f(w)=(f_\lambda(w))_{\lambda \in \Omega_{p}} \in \bA
$$
For each $\lambda \in \Omega_p$, we have the generating function $\mathcal{Z}_{\lambda}=\mathcal{Z}_{\lambda}(\hbar,q,(z_i)_{i\in\mathbb{Z}})$, as in Section 1.2, with the parameters $z_i$ relabeled since the corner box of $\lambda$ does not necessarily lie over the vertex $0$.
Define 
$$
c^{\lambda}_i := | \{\square \in \lambda \mid c(\square)=i \} | \ \ \  \text{and} \ \ \ \sigma_\lambda(i):=c^{\lambda}_{i-1}-c^{\lambda}_i
$$
where $c(\square)\in \{0,\ldots,n-1\}$ denotes the color of $\square$. Also, write 
\begin{equation*}
   \mu \prec_{\mathfrak{C}} \lambda \iff \lim_{w \to 0} \frac{f_{\mu}(w)}{f_{\lambda}(w)} = 0
  \end{equation*}
\begin{Theorem} \label{thm2} Let all notation be as above. Then we have
\begin{equation*}
    \begin{aligned}
     \textbf{V}_{p}(0_\mathfrak{C},\boldsymbol{z})= \prod_{\lambda \in \Omega_{p}}  \mathcal{Z}_{\lambda}\left(\hbar, q, \left( z_i^{\#} \right)_{i \in \mathbb{Z}} \right)
    \end{aligned}
\end{equation*}
where 
$$
z_i^{\#}=\left( \frac{\hbar}{q} \right)^{\nu_{\lambda}(i)}z_i
$$
and
\begin{equation*}
    \nu_{\lambda}(i)= \sum\limits_{\mu \prec_{\mathfrak{C}} \lambda} \sigma_{\mu}(i) + \sum\limits_{\lambda \prec_{\mathfrak{C}} \mu} \sigma_{\mu}(i+1) + \sum\limits_{\substack{\mu \prec_{\mathfrak{C}} \lambda \\ b_\mu=i}} 1 
    \end{equation*}
    and $b_\mu$ denotes the color of the corner box of $\mu$.
\end{Theorem}
\begin{proof}
This follows from a direct computation using (\ref{verpow}) in Section 2.4. 
\end{proof}
\subsection{}
A similar computation shows that the limit of the vertex function for affine type $A_n$ varieties likewise factorizes into a product of $\mathcal{Z}_{\lambda}$, with certain shifts of the parameters $z_i$. The difference in this case is that the presence of an oriented cycle in the quiver means that we must take the coloring of the paritions modulo $n$ and allow for arbitrarily wide partitions, provided the total number of boxes of a given color agrees with the dimension vector $\mathsf{v}$. Furthermore, the indices on $z_i$ must be taken modulo $n$.
 
\subsection{}

3d mirror symmetry, also known as symplectic duality, is a conjecture which associates to every symplectic variety $X$, a dual variety $X'$. See also \cite{MirSym2,MirSym1} for the recent developments of the 3d mirror symmetry conjectures. In the case where $X$ is a Nakajima quiver variety, 3d mirror symmetry provides the following data 
\begin{itemize}
    \item An isomorphisms of tori
\begin{equation} \label{3dmiriso}
\kappa: \, \bA'\times \bK'\times \matC_{\hbar'}^{\times} \times \matC^{\times}_{q} \ \ \longrightarrow \bA \times \bK \times \matC_{\hbar}^{\times} \times \matC^{\times}_{q}
\end{equation}
where $\bK'$ and $\bA'\times \matC_{\hbar'}^{\times}=\bT'$ are the K\"ahler and the equivariant tori of $X'$.

\item A bijection of sets of fixed points
\begin{equation*} \label{bij}
\textsf{b}: X^{\bT} \to  (X')^{\bT'}
\end{equation*} 
\end{itemize}
The restriction of $\kappa$ provides an isomorphism
\begin{equation*}
    \bar{\kappa}: \bA \to \bK', \ \ \ \bK \rightarrow \bA'
\end{equation*}

The map $\kappa$ also provides an identification of chambers with effective cones of the $3d$ mirror variety:
\begin{equation*}
d \bar{\kappa} (\mathfrak{C})=C_{\textrm{eff}}(X'), \  \ \ \ d \bar{\kappa} (C_{\textrm{eff}}(X))=\mathfrak{C}',
\end{equation*}
where $d\bar{\kappa}$ is the induced map of the Lie algebras and $C_{\textrm{eff}}(Y)\subset H^{2}(Y,\matR)$ is a certain cone associated to the choice of stability parameter for any Nakajima variety $Y$. It is therefore natural to think that the $3d$ mirror symmetry provides a pair $(X',\mathfrak{C}')$ for each pair $(X,\mathfrak{C})$. 

\subsection{}
Let $(X,\mathfrak{C})$ be a Nakajima variety of finite or affine $A_n$ type, and let $(X',\mathfrak{C}')$ be the variety and chamber related by $3d$-mirror symmetry. 
For $p\in X^{\bT}$ let $\textsf{b}(p)$ be the corresponding fixed point on the dual variety $X'$. Recall that $X^{\bT}$ is indexed by $(\mathsf{w}_0+ \ldots + \mathsf{w}_{n-1})$-tuples of colored partitions, where $\mathsf{w}\in \mathbb{N}^n$ is the framing dimension vector of $X$. The chamber $\mathfrak{C}'=d\bar{\kappa}(C_{\textrm{eff}}(X))$ provides a decomposition
$$
T_{\textsf{b}(p)} X' = N^{+}_{\textsf{b}(p)}\oplus N^{-}_{\textsf{b}(p)}.
$$
where $N^{+}_{\textsf{b}(p)},N^{-}_{\textsf{b}(p)}$ are the subspaces whose $\bA'$-characters take positive or negative values on $\mathfrak{C}'$, respectively.
We identify these spaces with their $K$-theory classes $N^{
\pm}_{\textsf{b}(p)} \in K_{\bT'}(pt)$. For $N \in K_{\bT'}(pt)$ given by a polynomial $N=w_1+\dots +w_m$ we abbreviate 
$$
\Xi(b,N)=\prod_{i=1}^m \prod_{n=0}^{\infty} \frac{1-b w_i q^n}{1-w_i q^n}
$$  
The following was conjectured in \cite{dinksmir}:

\begin{Conjecture}\label{conj}
{\it The vertex functions of $X$ with vanishing equivariant parameters are given by the Taylor series expansions of the following functions
\begin{equation*}\label{hypot}
\begin{array}{|c|}  
\hline \\
\ \ \ \kappa^{*} {\bf V}_p(0_{\mathfrak{C}},\zz)= \Xi(q/\hbar', (N^{-}_{\textsf{b}(p)})^{*} )\ \ \ \\
\\
\hline
\end{array}
\end{equation*}	
where $\kappa^{*}$ stands for substitution (\ref{3dmiriso}) and $(N^{-}_{\textsf{b}(p)})^{*}$ is the $\bT'$-module dual to $N^{-}_{\textsf{b}(p)}$ (i.e. the weights of $(N^{-}_{\textsf{b}(p)})^{*}$ are inverses of the weights of $N^{-}_{\textsf{b}(p)}$). }
\end{Conjecture}


\subsection{}
We give the proof of Conjecture \ref{conj} in the case of $X=X_{\lambda}$. The symplectic dual variety $X'$ can be described in the language of slices in affine Grassmannian (\cite{slices},\cite{slices2}) or as a bow variety (\cite{bow}). We proceed with the second option.  Our presentation here is not entirely self-contained, as we use the notations of \cite{bow}, in which case our choice of stability corresponds to $\nu_{\sigma}^{\mathbb{R}}=-1$. 

Bow varieties, like Nakajima quiver varieties, are formed as a GIT quotient of a space whose data is encoded by a combinatorial object, in this case a ``bow diagram." To obtain the bow diagram corresponding to $X_{\lambda}$, we start with the quiver encoding the data. We replace every edge of the quiver with an $o$, and every vertex of the quiver with an edge, labeled by the same number as the corresponding vertex in the quiver. We add an additional $o$ to the left and right ends of the diagram and also add segments extending past these $o$s. We label the segments on the ends by $0$. To account for the framing, we place an $x$ on the edge corresponding to the $0$th vertex, and we label each edge adjacent to this $x$ by the number that was previously on the edge before the $x$ was added. For example, see Figure \ref{bow}.

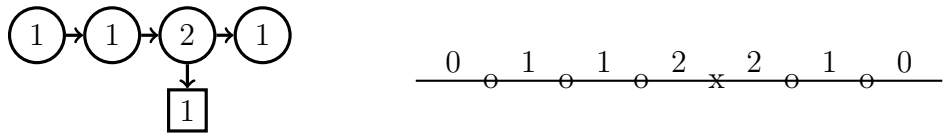
\begin{figure}[ht]
\begin{minipage}{.5\textwidth}
\centering
\begin{tikzpicture}[roundnode/.style={circle, draw=black, very thick, minimum size=5mm},squarednode/.style={rectangle, draw=black, very thick, minimum size=5mm}] 

\node[roundnode](1) at (1,-1){1};
\node[roundnode](0) at (0,-1){2};
\node[roundnode](-1) at (-1,-1){1};
\node[roundnode](-2) at (-2,-1){1};

\node[squarednode](f0) at (0,-2){1};

\draw[very thick,->] (-2)--(-1);
\draw[very thick,->] (-1)--(0);
\draw[very thick,->] (0)--(1);
\draw[very thick, ->] (0)--(f0);
\end{tikzpicture}
\end{minipage}
\begin{minipage}{.5\textwidth}
\centering
\begin{tikzpicture}[roundnode/.style={circle, draw=black, very thick, minimum size=5mm},squarednode/.style={rectangle, draw=black, very thick, minimum size=5mm}] 
\node at (5/2,-3/4){0};
\node at (2,-1){o};
\node at (3/2,-3/4){1};
\node at (1,-1){o};
\node at (1/2,-3/4){2};
\node at (0,-1){x};
\node at (-1/2,-3/4){2};
\node at (-1,-1){o};
\node at (-3/2,-3/4){1};
\node at (-2,-1){o};
\node at (-5/2,-3/4){1};
\node at (-3,-1){o};
\node at (-7/2,-3/4){0};
\draw[thick] (-4,-1)--(3,-1);
\end{tikzpicture}
\end{minipage}
\caption{The left diagram is the quiver data for the quiver variety $X_{\lambda}$ for $\lambda=(3,2)$. The diagram on the right encodes this data as a bow diagram.}\label{bow}
\end{figure}

To obtain the bow diagram of $X_{\lambda}'$, we simply interchange all $x$'s and $o$'s in the bow diagram of $X_{\lambda}$, and leave the dimension data and stability condition the same, see \cite{bow} Section 7.1.

\subsection{}
The torus action on a bow variety can be extracted from the bow diagram as follows. For each $x$, there is the data of linear maps (in notations of \cite{bow}):
$$
A_x: V_{x^-} \to V_{x^+}, \quad  B_{\pm} \in Hom(V_{x^\pm},V_{x^\pm}), \quad b_x:V_{x^-} \to \mathbb{C}, \quad a_x:\mathbb{C} \to V_{x^+}$$
where $V_{x^-}$ and $V_{x^+}$ are vector spaces fixed by the choice of dimension. There is a $\mathbb{C}^\times$-action on this data as follows:
$$
u_x \cdot (A_x,B_{x^-},B_{x^+},a_x,b_x) = (A_x,B_{x^-},B_{x^+},a_x u_x^{-1},u_{x} b_x)
$$
where $u_x$ denotes the coordinate on $\mathbb{C}^\times$. We obtain a copy of $\mathbb{C}^\times$ for every $x$ in the bow diagram.

There is an additional $\mathbb{C}^\times$ action, which corresponds to the action of the parameter $\hbar$, but we will ignore this here for the sake of simplicity.

In addition, the (extended) K\"ahler torus for a bow variety can be extracted from the bow diagram as 
$$
\widetilde{\bK}= \left(\mathbb{C}^\times\right)^{\text{number of } o's}
$$
We use the term ``extended K\"ahler torus", since $\widetilde{\bK}$ fits into a natural exact sequence
$$
1 \longrightarrow \mathbb{C}^\times \longrightarrow \widetilde{\bK} \longrightarrow \bK \longrightarrow 1
$$
where $\bK$ is the usual K\"ahler torus. The map $\widetilde{\bK} \to \bK$ is given by $(\zeta_{i}) \mapsto (\frac{\zeta_{i-1}}{\zeta_{i}})$, where $\zeta_{i}$ are the coordinates on $\widetilde{\bK}$,  In other words, the usual K\"ahler parameters are given by $z_i=\frac{\zeta_{i-1}}{\zeta_{i}}$.
Since the 3d-mirror variety can be obtained by interchanging $x$'s and $o$'s, this description of the equivariant and K\"ahler tori provides a natural identification of $\bA \cong \widetilde{\bK'}$ and $\widetilde{\bK} \cong \bA'$.

\subsection{}
We fix notation as follows. Let $\lambda$ be a partition, with associated dimension vector $\mathsf{v}=(\mathsf{v}_{-r},\ldots,\mathsf{v}_s)$ and quiver variety $X_{\lambda}$. In the bow diagram, there are $r+1$ $o$'s, followed by one $x$, followed by $s+1$ $o$'s. We label the coordinates on the K\"ahler torus as $\zeta_{i}$, $i=-r-1,\ldots,s$.

Interchanging $x$'s and $o$'s gives an obvious identification of $\zeta_{i}$ as coordinates on $\bA'$. 

\begin{Proposition}
There is a unique $\widetilde{\bK}$-fixed point $p'$ on the variety $X_{\lambda}'$. The $\widetilde{\bK}$ character of the tangent space of $X_{\lambda}'$ at $p'$ is
$$
T_{p'}X_{\lambda}' = \sum_{\square \in \lambda} z_{\square} |_{\hbar=q} + z_{\square}^{-1} |_{\hbar=q} \in K_{\widetilde{\bK}}(pt)
$$
where $z_{\square}$ is as in section 1.4 and $z_{i}=\frac{\zeta_{i-1}}{\zeta_{i}}$, and the notation $z_{\square}|_{\hbar=q}$ indicates the substitution of $\hbar=q$ into the expression for $z_{\square}$.
\end{Proposition}
\begin{proof}
As mentioned above, we use the notations of \cite{bow}. We have vector spaces $V_{i}$, $i=-r-1,\ldots,s$ and in particular, $\dim V_{-1}= \dim V_{0} = \mathsf{v}_0$. Let $(A_i,B_i,a_i,b_i,C,D)$ be a representative of a $\widetilde{\bK}$-fixed point on $X_{\lambda}'$, where $A_i\in Hom(V_{i},V_{i+1})$ for $i\neq -1$, $C\in Hom(V_{-1},V_{0})$, $D\in Hom(V_{0},V_{-1})$. By \cite{bow} Proposition 2.8, our choice of stability implies that there is no proper subspace of $\bigoplus_{i} V_i$ containing the image of all $a_i$ closed under the action of all $A_i$, $C$, and $D$ and closed under taking preimages under each $A_i$. By definition of the $\widetilde{\bK}$ action, this implies that the $\widetilde{\bK}$-weights of the tautological bundles $\mathcal{V}_j$ at the fixed point are just $\zeta_{j_1},\ldots, \zeta_{j_{\mathsf{v}_j}}$ where each $j_k \in \{-r-1,\ldots,s\}$. The conditions $\textbf{(S1)}$ and $\textbf{(S2)}$ from \cite{bow} imply that $A_i$ is injective for $i<-1$ and is surjective for $i\geq 0$. This means that if $i<0$, the $\widetilde{\bK}$-weights of $\mathcal{V}_i$ at the fixed point are the same as the weights of $\mathcal{V}_{i+1}$, assuming that $\mathsf{v}_{i+1}=\mathsf{v}_{i}$. If $\mathsf{v}_{i+1}=\mathsf{v}_{i}+1$, then $\mathcal{V}_i$ has an additional weight of $\zeta_i$. These are the only two options, as the dimensions arise from $\lambda$, a partition. A similar argument gives the weights of $\mathcal{V}_i$ for $i\geq 0$ at the fixed point. In fact, this also proves that there is exactly one fixed point.

The discussion in Section 2.5 of \cite{bow}, interpreted with respect to $\widetilde{\bK}$-equivariance, describes the tangent space at the fixed point in terms of virtual bundles associated to the tautological bundles. From this, a straightforward calculation shows that 

$$
T_{p'}X_{\lambda}' = \sum_{\substack{i \in A \\ j \in B}} \frac{\zeta_i}{\zeta_j} + \frac{\zeta_j}{\zeta_i}
$$
where $A=\{i\geq 0 \mid \mathsf{v}_{i}-\mathsf{v}_{i+1}=1\}$ and $B=\{i< 0 \mid \mathsf{v}_{i+1}-\mathsf{v}_{i}=1\}$.

Since $z_i=\frac{\zeta_{i-1}}{\zeta_i}$, a straightforward calculation shows that this is exactly the statement of the proposition.
\end{proof}

Ignoring the action of $\hbar'$, the previous Proposition, along with Theorem \ref{mthm} implies Conjecture \ref{conj} in the case of $X_{\lambda}$. 
\subsection{}
For an arbitrary type $A$ quiver variety $X$ and a fixed point $p$ indexed by partitions $(\lambda_i)$, the argument above, along with the explicit bijection on fixed points (\cite{futurebow}, \cite{naksatake}), implies that the $\widetilde{\bK}$-character of the tangent space $T_{b(p)}X'$ at the corresponding fixed point is 
$$T_{b(p)} X' =\sum_{\lambda \in \Omega_p} \sum_{\square \in \lambda} z_{\square} |_{\hbar=q} + \sum_{\lambda \in \Omega_p} \sum_{\square \in \lambda} z_{\square}^{-1} |_{\hbar=q} \in K_{\widetilde{\bK}}(pt)$$
Along with Theorem \ref{thm2}, this proves Conjecture \ref{conj} in this case.

\printbibliography
\newpage

\vspace{12 mm}

\noindent
Hunter Dinkins\\
Department of Mathematics,\\
University of North Carolina at Chapel Hill,\\
Chapel Hill, NC 27599-3250, USA\\
hdinkins@live.unc.edu

\vspace{3 mm}

\noindent
Andrey Smirnov\\
Department of Mathematics,\\
University of North Carolina at Chapel Hill,\\
Chapel Hill, NC 27599-3250, USA;\\
Steklov Mathematical Institute\\
of Russian Academy of Sciences,\\
Gubkina str. 8, Moscow, 119991, Russia

\end{document}